\documentclass[12pt]{amsart}

\usepackage{amssymb,amsmath,graphics,verbatim}
\usepackage{latexsym}
\usepackage{eucal}
\usepackage{a4wide}

\usepackage{times}
\usepackage[all,pdf]{xy}\SelectTips {cm}{}

\newtheorem{theorem}{Theorem}
\newtheorem{lemma}[theorem]{Lemma}
\newtheorem{proposition}[theorem]{Proposition}
\newtheorem{corollary}[theorem]{Corollary}
\theoremstyle{definition}
\newtheorem{definition}[theorem]{Definition}
\theoremstyle{remark}

\newcommand{\Ad}{\mathrm{Ad}}
\newcommand{\ad}{\mathrm{ad}}
\newcommand{\rz}{\mathbb{R}}
\newcommand{\kz}{\mathbb{K}}
\newcommand{\cz}{\mathbb{C}}

\newcommand{\zz}{\mathbb{Z}}
\newcommand{\sz}{\mathbb{S}}

\newcommand{\cs}{\mathfrak{cs}}

\newcommand{\CS}{\mathfrak{CS}}

\newcommand{\cM}{\mathcal{M}}

\newcommand{\gl}{\mathrm{gl}}

\newcommand{\SO}{\mathrm{SO}}
\newcommand{\so}{\mathrm{so}}

\newcommand{\tr}{\mathrm{tr}}
\newcommand{\vol}{\mathrm{vol}}
\newcommand{\dd}{\mathrm{d}}
\newcommand{\GL}{{\mathrm{GL}}}

\newcommand{\SL}{\mathrm{SL}}
\newcommand{\End}{\mathrm{End}}
\newcommand{\GLnR}{\GL_n(\rz)}
\newcommand{\GLnC}{\GL_n(\cz)}
\newcommand{\Un}{\mathrm{U}(n)}
\newcommand{\SU}{\mathrm{SU}}
\newcommand{\U}{\mathrm{U}}

\newcommand{\Stab}{\mathrm{Stab}}

\begin{document}
\title[Invariants of stably trivial vector bundles]{Invariants of stably trivial vector bundles with connection}
\author{Sergiu Moroianu}
\address{Universitatea Bucure\c sti \\ Facultatea de Matematic\u a \\strada Academiei 14, and
Institutul de Matematic\u a al Academiei Rom\^ane, calea Grivi\c tei 21, Bucharest, Romania}
\email{moroianu@alum.mit.edu}
\date{\today}
\begin{abstract}
We define a Chern--Simons invariant of connections on stably trivial vector bundles over smooth manifolds, taking values in $3$-forms modulo closed forms with integral cohomology class.
We show an additivity property of this invariant for connections defined on a direct sum of bundles, under a certain block-diagonality condition on the curvature. As a corollary, we deduce an obstruction for conformally immersing a $n$-dimensional Riemannian manifold in a translation manifold of dimension $n+1$.
\end{abstract}
\maketitle

\section{Introduction}
By a striking result of Chern and Simons \cite{cs}, the Chern--Simons invariant of a compact oriented Riemannian $3$-manifold $(X,g)$ conformally immersed in the flat Euclidean space $\rz^4$ must vanish modulo $\zz$, showing, for instance, that the real projective space $\rz P^3$ cannot be conformally immersed in $\rz^4$. This result was extended in a recent paper by \v{C}ap, Flood and Mettler \cite{cfm} to equiaffine immersions, a notion that we recall below. Motivated by those results, we prove in this note a more general vanishing result for the 
Chern--Simons invariant of a stably trivial real vector bundle with connection over an arbitrary smooth manifold (Theorem \ref{tma}). When the bundle is the tangent bundle of a Riemannian manifold, we derive an obstruction for conformal immersions as hypersurfaces in translation manifolds:

\begin{theorem}\label{gencsd}
Let $f:X^n\to Y^{n+1}$ be a conformal immersion of an oriented Riemannian manifold 
$(X,g)$ of dimension $n$ in a $n+1$-dimensional translation manifold $Y$. Then, for every smooth singular $3$-cycle $C\in C_3(X,\zz)$, the evaluation of 
$\CS(\nabla^X)$ on $C$ must vanish modulo integers:
\[\int_C\CS(\nabla)\in\zz.\]
\end{theorem}
 The definition of a translation manifold is recalled in Section \ref{aci}. We should stress that $X$ is not assumed to be parallelizable; to define the Chern--Simons invariant of $\nabla^X$ it suffices that $TY$ be trivial, which entails that $TX$ is stably trivial. As a consequence, we deduce, for instance, that $3$-dimensional lens spaces cannot be conformally embedded in any (possibly incomplete) Riemannian $4$-manifold with trivial global holonomy. 

We first review Chern--Simons invariants of trivial bundles using a principal-bundle-free formalism, based on the calculus of matrix-valued forms. We then extend the definition to stably trivial bundles with connection; prove the additivity property; and finally deduce various obstructions to immersions into translation manifolds. In the Appendix we show that a certain $3$-form on the general linear group, constructed in terms of the Maurer--Cartan form, is integral, and we compute the Riemannian Chern--Simons invariant of $3$-dimensional lens spaces.

For several reasons, we do not treat here higher-dimensional Chern--Simons forms: first, because a large part of the current interest in Chern--Simons invariants seems to concentrate in dimension $3$; secondly, since the tangent bundle of oriented manifolds is trivial in dimension $3$, but in general not stably trivial in higher dimensions; and thirdly, because there are technical difficulties in extending the additivity property (Lemma \ref{main}) beyond dimension $3$. 
With these limitations, the paper is essentially self-contained, and could hopefully also be used as an efficient introduction to the subject, free of the customary reliance on the topology of classifying spaces.

\section{Matrix-valued forms}

Let $X$ be a smooth manifold. Denote by $\cM_{kk'}$ the space of real matrices with $k$ rows and $k'$ columns, and let $X\times \cM_{kk'}\to X$ be the trivial vector bundle. A \emph{matrix-valued form} on $X$ is a section in the tensor product bundle $\Lambda^*X\otimes_{\rz} \cM_{kk'}$. For $k=k'=1$ this is the usual algebra of differential forms on $X$. 

Besides the evident associative product (denoted below by concatenation)
\[(\Lambda^pX\otimes \cM_{kk'})\times (\Lambda^qX\otimes \cM_{k'k''})\mapsto \Lambda^{p+q}X\otimes \cM_{kk''},
\]
there are two other natural operations on matrix-valued forms. 

First, we have the first-order differential operator $\dd$, the exterior differential on $X$ twisted by the trivial connection (also denoted  $\dd$) on the trivial vector bundle 
$\cM_{kk'}$: \begin{align*}
\dd:\Lambda^pX\otimes \cM_{kk'}\to \Lambda^{p+1}X\otimes \cM_{kk'}, &&\dd(\alpha\otimes A)= (\dd \alpha)\otimes A+(-1)^p\alpha\dd A.
\end{align*}
Here $\alpha\in \Lambda^pX$ is a form and $A:X\to\cM_{kk'}$ is a section in $X\times \cM_{kk'}$. The operator $\dd$ has the signed Leibniz property
\[\dd(uv)= \dd(u)v+(-1)^{\deg(u)} u\dd(v),
\]
where, by definition, the degree of a matrix-valued form $u\in \Lambda^pX\otimes \cM_{kk'}$ is $\deg(u)=p$.

The second natural operation is the trace, defined when $k=k'$:
\begin{align*}
\tr:\Lambda^pX\otimes \gl_k(\rz)\to \Lambda^{p}X, &&\tr(\alpha\otimes A)= \alpha\tr(A).
\end{align*}
Although matrix-valued forms do not form a super-commutative algebra, their trace has the sign-commutation property
\[\tr(uv)=(-1)^{pq}\tr(vu).
\]
for every $u\in \Lambda^pX\otimes \cM_{kk'}$ and $v\in \Lambda^qX\otimes \cM_{k'k}$.

Combining these two operations, it is useful to note the identity 
\[\dd(\tr(u))=\tr(\dd(u)).\]

The same properties hold for complex matrix-valued forms. We shall distinguish between the two cases by a subscript in the complex trace $\tr_\cz$, in order to distinguish it from its real counterpart.

\section{Chern--Simons forms and invariants}

Let $E$ be a topologically trivial real
vector bundle of rank $n$ with connection $\nabla$ over $X$. Once we fix a frame in $E$, we write $\nabla=\dd+\omega$ for some matrix-valued $1$-form $\omega$, and we define the Chern--Simons $3$-form by
\[
\cs(\omega)=-\tfrac{1}{16\pi^2} \tr\left(\omega\dd \omega +\tfrac{2}{3}\omega^3\right)\in\Lambda^3(X).
\]
Denote by $\Omega=\dd\omega+\omega^2\in \Lambda^2(X,\gl_n\rz)$ the matrix-valued curvature $2$-form of $\nabla$ in the fixed frame. 
One can alternately describe the Chern--Simons form as
\[\cs(\omega)=-\tfrac{1}{16\pi^2} \tr\left(\omega\Omega -\tfrac{1}{3}\omega^3\right).
\]
A standard computation gives
\[\dd (\cs(\omega)) = -\tfrac{1}{16\pi^2}\tr(\Omega^2)\in\Lambda^4(X).
\]

\begin{definition}
A complex $k$-form $\omega\in\Lambda^k(X)\otimes_\rz \cz$ on a smooth manifold $X$ is said to be \emph{integral} if it is closed, and its cohomology class belongs to the image of the natural map $H^k(X,\zz)\to H^k(X,\cz)$. 
The space of integral $k$-forms on $X$ is denoted $\Lambda^k_\zz(X)$.
\end{definition}
Equivalently, $\omega\in\Lambda^k_\zz(X)$ if and only if $\int_C\omega\in\zz$ for every oriented smooth singular $k$-cycle $C$ with integer coefficients.

In a different trivialization obtained by (right-)multiplying the given frame in $E$ by a map $a:X\to \GLnR$, the connection $1$-form becomes $\omega'=a^{-1}\omega a+a^{-1}\dd a$. By an easy computation using the properties of matrix-valued forms, the associated Chern--Simons form changes into
\begin{equation}\label{csop}
\cs(a^{-1}\omega a+a^{-1}\dd a)=\cs(\omega) 
+\tfrac{1}{48\pi^2} \tr\left((a^{-1}\dd a)^3\right) -\tfrac{1}{16\pi^2}  \dd( \dd a a^{-1}\omega) \in\Lambda^3(X).
\end{equation}
The form $a^{-1}\dd a$ is the pull-back by $a$ of the Maurer--Cartan form $\mu\in\Lambda^1(\GLnR,\gl_n\rz)$. It is a fact, detailed in Appendix \ref{apa} for the sake of completeness, that for every $n\geq 3$, the form $\frac{1}{48\pi^2}\tr(\mu^3)$ is closed, defines an integral cohomology class on $\GLnR$, and its evaluation on the $3$-homology cycle 
$\SO(3)\subset \GLnR$ equals $-1$, thus motivating the choice of the normalization constant. 
Equation \eqref{csop} thus shows that $\cs(\omega')$
agrees with $\cs(\omega)$ up to an integral $3$-form. It follows that the class of $\cs(\omega)$ modulo $\Lambda^3_\zz(X)$ gives rise to a well-defined element $\CS(\nabla)\in\Lambda^3(X)/\Lambda^3_\zz(X)$, independent of the trivialization, that we call the \emph{Chern--Simons invariant} of $\nabla$. Whenever the first Pontriaghin form $\tr(\Omega^2)$ of $\nabla$ vanishes, $\cs(\omega)$ is closed, and its cohomology class is independent of the choice of trivialization up to an integral class.

\subsection{Vanishing results}
The vanishing of a Chern--Simons invariant means that it is represented by a closed Chern--Simons form whose associated de Rham cohomology class is integral, i.e., it lives in the image of the natural map $H^3(X,\zz)\to H^3(X,\cz)$. The above discussion ensures that the vanishing of a Chern--Simons invariant is independent of the choice of trivialization used for defining the connection $1$-form, so it is a property of the connection, rather than of the connection \emph{form} in some trivialization.

Here are some evident instances of vanishing results:
\begin{lemma}\label{lvan}
\begin{enumerate}
\item The Chern--Simons invariant of a trivial connection vanishes.
\item The Chern--Simons form of a flat connection on a trivial vector bundle of rank $1$ vanishes.
\item If $\omega\in \Lambda^1(X)$ defines an involutive distribution $\ker(\omega)\subset TX$ of codimension $1$ in $X$, i.e., $\omega \wedge \dd \omega=0$, then the Chern--Simons form of the connection $\dd+\omega$ on the trivial bundle $\rz\times X$ vanishes.
\end{enumerate}
\end{lemma}
\begin{proof}
By definition, in a parallel frame, the connection $1$-form of a trivial connection vanishes, so both terms in the Chern--Simons form vanish as well.

If the rank of the bundle is $1$, the term $\omega^3$ vanishes identically by anti-symmetry. The term $\omega\Omega$ also vanishes if the curvature is $0$, proving the second claim. 

The third claim is similar, in that the term $\Omega=\dd\omega$ does not need to vanish anymore, however the Frobenius integrability condition implies
$\omega\Omega=0$.
\end{proof}

\subsection{Complex Chern--Simons forms and invariants}
For trivial complex vector bundles endowed with a $\cz$-linear connection, more or less the same definition makes sense using the complex trace, resulting in a complex-valued $3$-form on $X$. In that case, the form $\frac{1}{24\pi^2}\tr_\cz(\mu^3)$ is integral on $\GLnC$ by Lemma \ref{lemintc}, so the natural normalization constant for the Chern--Simons invariant is in this case
$1/8\pi^2$:
\[
\cs_\cz(\omega)=-\tfrac{1}{8\pi^2} \tr_\cz\left(\omega\dd \omega +\tfrac{2}{3}\omega^3\right)\in\Lambda^3(X)\otimes_\rz\cz.
\]
This form has the same invariance properties as its real counterpart, eventually defining a class in 
$\CS_\cz(\nabla) \in(\Lambda^3(X)\otimes\cz)/\Lambda^3_\zz(X)$ independent of trivializations, called the complex Chern--Simons invariant of $\nabla$.

Since every complex vector bundle has a subjacent real structure, it is natural to inquire about a possible relationship between the real and the complex Chern--Simons invariants of a trivial complex vector bundle with complex connection. This relation is simply:
\begin{lemma}
$\CS(\nabla)=\Re(\CS_\cz(\nabla))$.
\end{lemma}
\begin{proof}
For $A\in\End_\cz(\cz^n)$ define $A_\rz\in\End_\rz(\cz^n)$ to be the image of $A$ through the natural inclusion $\End_\cz(\cz^n)\hookrightarrow\End_\rz(\cz^n)$.
It is evident that the inclusion is an algebra map, and that $\tr(A_\rz)=2\Re(\tr_\cz(A))$. Then 
\[\left(\omega\dd \omega +\tfrac{2}{3}\omega^3\right)_\rz= \left(\omega_\rz\dd \omega_\rz +\tfrac{2}{3}\omega_\rz^3\right).
\]
The different normalization constants in the definitions of $\cs$ and $\cs_\cz$ account for the factor $2$ relating the real and the complex traces. 
\end{proof}
 
By tensoring with $\cz$, every trivial real vector bundle $E$ with connection gives rise to a complex vector bundle $E_\cz$ endowed with a complex connection. Any trivialization of $E$ will induce a complex trivialization of $E\otimes_\rz\cz$, and the connection forms in $E$ and $E_\cz$ are the same $\gl_n(\rz)$-valued $1$-form $\omega$.
From the difference in normalization constants, it follows that 
$\cs(\omega)=\frac{1}{2}\cs_\cz(\omega)$.

\section{Stable Chern--Simons invariants}
Let $E_1,E_2$ be two trivial vector bundles (real or complex) over $X$ with connections $\nabla_1,\nabla_2$. We say that a connection $\nabla$ on $E:=E_1\oplus E_2$ \emph{extends} $\nabla_1$ and $\nabla_2$ if
\begin{align}\label{conext}
\Pi_1\nabla \Pi_1=\nabla_1,&& \Pi_2\nabla \Pi_2=\nabla_2
\end{align}
where $\Pi_1,\Pi_2$ are the projectors on $E_1$, respectively $E_2$, in $E$. 
If we fix trivializations of $E_1$ and $E_2$, we write $\nabla_1=\dd+\omega_1$, $\nabla_2=\dd+\omega_2$, and therefore $\nabla=\dd+\omega$, where
\[\omega=\begin{bmatrix} \omega_1 & A\\B&\omega_2
\end{bmatrix}.
\]
Here $A$, respectively $B$ are some matrix-valued $1$-forms of suitable dimensions.
Denote by $R$ the curvature of $\nabla$, and by $\Omega$ its matrix-valued $2$-form in the fixed trivialization.
For $i,j\in\{1,2\}$ denote by $R_{ij}$ the $E_j\to E_i$ component of the curvature:
\[R_{ij} = \Pi_i R \Pi_j.
\]
\begin{lemma}\label{main}
If $R_{12}=0$ and $R_{21}=0$, then in the real case
\[\cs(\omega)=\cs(\omega_1)+\cs(\omega_2),
\]
while in the complex case $\cs_\cz(\omega)=\cs_\cz(\omega_1)+\cs_\cz(\omega_2)$.
\end{lemma}
The proof is formally the same in both cases, so we disregard the subscript $\cz$ in the trace whenever the bundles are complex.
\begin{proof}
Compute the various ingredients of the curvature and of the Chern--Simons form of $\omega$ as follows:
\begin{align*}
\omega^2 ={}& \begin{bmatrix} \omega_1^2 +AB& \omega_1 A+A\omega_2\\B\omega_1+\omega_2B&\omega_2^2+BA
\end{bmatrix}\\
\Omega ={}& \begin{bmatrix} \dd \omega_1+\omega_1^2 +AB& \dd A+\omega_1 A+A\omega_2
\\ \dd B+ B\omega_1+\omega_2B&\dd \omega_2+\omega_2^2+BA
\end{bmatrix}\\
\omega\Omega = {}&
\begin{bmatrix} \omega_1\dd \omega_1+\omega_1^3 +\omega_1AB +A\dd B
+AB\omega_1+A\omega_2 B& \!\!\!\!\!\!\!\! *
\\ \!\!\!\!\!\!\!\!\!\!\!\!\!\!\!\!\!\!\!\!\!\!\!\!\!\!\!\!\!\!\!\!\!\!\!\!\!\!\!\!\!\!\!\!\!\!\!\!  * & \!\!\!\!\!\!\!\!\!\!\!\!\!\!\!\!\!\!\!\!\!\!\!\!\!\!\!\!\!\!\!\!\!\!\!\!\!\!\!\!\!\!\!\!\!\!\!\!\!\!\!
B\omega_1 A+B\dd A+BA\omega_2+\omega_2\dd \omega_2+\omega_2^3+\omega_2BA
\end{bmatrix}\\
\omega^3 ={}& \begin{bmatrix} \omega_1^3 +\omega_1AB+AB\omega_1+A\omega_2 B& 
*\\
*&
B\omega_1 A+\omega_2BA+\omega_2^3+BA\omega_2
\end{bmatrix}.
\end{align*}
We have omitted the off-diagonal terms in $\omega \Omega$ and in $\omega^3$, replacing them with the symbol $*$, as they are eventually irrelevant for the trace. From the above, we compute the traces appearing in the Chern--Simons form:
\begin{align*}
\tr(\omega^3) ={}& \tr(\omega_1^3) +\tr(\omega_2^3) + 3\tr(\omega_1AB)+3\tr(\omega_2BA)\\
\tr(\omega\Omega) = {}& \tr(\omega_1\Omega_1)+\tr(\omega_2\Omega_2)+\tr(A\dd B)+\tr(B\dd A)
+3\tr(\omega_1AB) +3\tr(\omega_2BA)\\
\cs(\omega)={}&\cs(\omega_1)+\cs(\omega_2)+\tr[A(\dd B+B\omega_1+\omega_2 B)]
+\tr[B(\dd A+\omega_1 A+A\omega_2)].
\end{align*}
We now recognise in the last two terms the off-diagonal terms in the curvature: 
$\Omega_{21}=\dd B+B\omega_1+\omega_2 B$, 
$\Omega_{12}=\dd A+\omega_1A+A\omega_2$.
These terms vanish by hypothesis, proving our claim.
\end{proof}

A similar result appears in \cite{Mil74} for the Simons invariants of complete metrics of constant positive sectional curvature. Note also a related \emph{multiplicative} statement for Cheeger-Simons differential characters in \cite[Theorem 4.7]{CheegSim73}. 

As a consequence, we get an additivity property of Chern--Simons invariants of trivial bundles with connections. More specifically, if $E_1,E_2$ are trivial bundles with connections and $\nabla$ is a connection on $E:=E_1\oplus E_2$ extending $\nabla_1$ and $\nabla_2$ as in \eqref{conext}, then, whenever the curvature endomorphism of $\nabla$ is diagonal with respect to the splitting, we have
\[\CS(\nabla)=\CS(\nabla_1)+\CS(\nabla_2) \mod\Lambda^3_\zz(X).
\]
We shall prove a stable version of the above identity shortly. Let us recall that a \emph{stably trivial vector bundle} is a bundle which admits a trivial complement in a trivial bundle. More precisely, $E$ is stably trivial if there exists $E'$ trivial with $E\oplus E'$ also trivial.

\begin{proposition}\label{inde}
Let $E\to X$ be a stably trivial real (respectively complex) vector bundle with connection $\nabla^E$, and $E'$ a trivial complement (i.e.,  such that $E\oplus E'$ is trivial) endowed with a trivial connection $\nabla'$. Then the invariant $\CS(\nabla\oplus\nabla')$ (respectively $\CS_\cz(\nabla\oplus\nabla')$) is independent of 
the choice of $(E',\nabla')$.
\end{proposition}
\begin{proof}
We only treat the real case, the proof in the complex case being formally identical.
Choose another trivial complement $E''$ with trivial connection $\nabla''$ such that $E\oplus E''$ is trivial. 

Notice that $\nabla\oplus\nabla'\oplus \nabla''$ is the direct sum connection on $E\oplus E'\oplus E''$ with respect to the splitting in the direct sum of the trivial bundles $E\oplus E'$ and $E''$, so in particular its curvature is diagonal with respect to this splitting. We use Lemma \ref{main} to deduce that, once we fix trivializations of $E\oplus E'$ and of $E''$,
\[\cs(\nabla\oplus\nabla'\oplus \nabla'')=\cs(\nabla\oplus\nabla')+\cs(\nabla'').\]
But the last term vanishes by Lemma \ref{lvan} since $\nabla''$ is trivial, hence we have proved the equality $\cs(\nabla\oplus\nabla'\oplus \nabla'')=\cs(\nabla\oplus\nabla')$ in the chosen trivializations. Passing to the residue modulo $\Lambda^3_\zz(X)$, we get 
$\CS(\nabla\oplus\nabla'\oplus \nabla'')=\CS(\nabla\oplus\nabla')$.

Similarly, fixing trivializations of $E\oplus E''$ and of $E'$, we have $\cs(\nabla\oplus\nabla'\oplus \nabla'')=\cs(\nabla\oplus\nabla'')$, and we deduce that independently of trivializations, 
$\CS(\nabla\oplus\nabla'\oplus \nabla'')=\CS(\nabla\oplus\nabla'')$. 

We deduce the equality $\CS(\nabla\oplus\nabla')=\CS(\nabla\oplus\nabla'')$.
\end{proof}

\begin{definition}
The \emph{Chern--Simons invariant} $\CS(\nabla)$ of a stably trivial vector bundle $E$ with connection 
$\nabla$ is defined as the class $\CS(\nabla\oplus\nabla')\in \Lambda^3(X)/\Lambda^3_\zz(X)$, for any trivial vector bundle $E'$ endowed with a trivial connection $\nabla'$ such that $E\oplus E'$ is trivial. (By Proposition \ref{inde}, that class is independent of the choices involved.)
\end{definition}

\section{Conformal invariance}
If $(X^3,g)$ is a parallelizable compact $3$-Riemanian manifold, the Chern--Simons invariant of the Levi-Civita connection is known to be conformally invariant \cite{cs}. We extend this result when the tangent bundle $TX$ is only \emph{stably} trivial as follows:
\begin{proposition}\label{confinv}
Let $(X^n,g)$ be a Riemannian manifold of dimension $n$ with stably trivial tangent bundle, and $\nabla$ its Levi-Civita connection. Fix a trivial complement of rank $k$, $L$, to $TX$, with its trivial connection $\dd$, and a global frame $s$ in the trivial bundle $TX\oplus L$ over $X$. Let 
$\omega\in C^\infty(X,\gl_{n+k})$ be the connection $1$-form of $\nabla\oplus \dd$ in the fixed framing.
Then the Chern--Simons $3$ form $\cs(\omega)$ is conformally invariant up to exact forms on $X$.
\end{proposition}
\begin{proof}
Let $f\in C^\infty(X)$ be a conformal factor, $\nabla^t$ the connection associated to the conformal metric $g_t=e^{2tf}g$, and $\omega_t\in C^\infty(X,\gl_{n+k})$ the connection $1$ form in the fixed framing. 
An easy computation valid for for every $1$-parameter variation of the Chern--Simons form gives:
\begin{align*}
\partial_t\cs(\omega_t)_{|t=0}  = {}& \tr(\dot\omega_t\dd\omega_t+\omega_t\dd\dot\omega_t +2\dot\omega_t\omega_t^2)_{|t=0} \\
={}& \dd\tr(\dot\omega\omega) +2\tr(\dot\omega(\dd\omega+\omega^2)).
\end{align*}
The first term is exact. In the second term, $\dd\omega+\omega^2$ is the matrix of the curvature of the connection $\nabla\oplus \dd$ in the fixed framing $s$. A direct application of Koszul's formula shows, in our case of conformal variations, that the variation of the Levi-Civita connection is
\[\partial_t(\nabla^t_UV) = U(f)V+V(f)U-g(U,V)\nabla f,
\]
in particular it is independent of $t$. So $\dot\omega$ is the matrix in the fame $s$ of the 
block-diagonal endomorphism (with respect to the splitting $TX\oplus L$)
\[\begin{bmatrix}\partial_t(\nabla^t)_{|t=0} &0\\0&0
\end{bmatrix}\]
One can compute $\tr(\dot\omega(\dd\omega+\omega^2))$ in any local basis, that we now choose to be compatible with the splitting:
\[\tr(\dot\omega(\dd\omega+\omega^2))=\tr\left(\partial_t(\nabla^t)_{|t=0} R\right)
\]  
where $R$ is the curvature endomorphism of $\nabla$. So the bundle $L$ does not contribute at all to this term. In any orthonormal frame $\{e_j\}$ on $X$,
\begin{align*}\partial_t(\nabla^t)_{|t=0} ={}&\dd f\otimes I +\sum_i e^i\otimes \left(\dd f\otimes e_i - e^j\otimes\nabla f\right)\\
R={}&\tfrac{1}{4}\sum_{j,k,\alpha,\beta}R_{jk\alpha\beta}e^j\wedge e^k\otimes\left(e^\alpha\otimes e_\beta-e^\beta\otimes e_\alpha\right).
\end{align*} 
Using the anti-symmetry of the Riemann curvature coefficients, we get $\tr((\nabla f\otimes I) R)=0$.
The remaining term gives
\[\sum_{i,j,k,l}R_{jkil}\partial_l(f),
\]
which vanishes by the first Bianchi identity.
\end{proof}

\section{Vanishing results}

\begin{theorem}\label{tm}
Let $E\to X$ be a stably trivial vector bundle endowed with a connection $\nabla^E$. Assume that 
there exists a line bundle $L\to X$ with a flat connection $\nabla^L$, and a trivial connection 
$\nabla$ on the bundle $E\oplus L$ such that
\begin{align*}
\nabla^E= \Pi_E\nabla \Pi_E, && \nabla^L= \Pi_L\nabla \Pi_L,
\end{align*}
where $\Pi_E$ and $\Pi_L=1-\Pi_E$ are the projectors on the $E$, respectively $L$ factors in $E\oplus L$.
Then $\CS(\nabla^E)=0$.
\end{theorem}
\begin{proof}
The conditions of Lemma \ref{main} are clearly satisfied since $\nabla$ is trivial. Moreover, $\CS(\nabla)=0$ by the same reason. Also, $\CS(\nabla^L)=0$ by Lemma \ref{lvan}. It follows from Lemma \ref{main} that $\CS(\nabla^E)=0$.
\end{proof}

From this, we recover some classical, and some more recent, vanishing results modulo $\zz$.

\begin{corollary}[Chern--Simons \cite{cs}]
Let $f:X^3\to \rz^4$ be a conformal immersion of an oriented Riemannian compact $3$-fold 
$(X,g)$ in the flat Euclidean space $\rz^4$, and $\omega$ the connection $1$-form of the Levi-Civita connection of $g$ in some global frame on $X$. Then 
\[\int_X\cs(\omega)\in\zz.\]
\end{corollary}
\begin{proof}
First, assume that $f$ is an isometric immersion. The tangent bundle of every orientable $3$-manifold is trivial \cite{sti}, so it makes sense to consider the Chern--Simons form for the Levi-Civita connection in some trivialization. The pull-back of $T\rz^4$ through the isometric immersion is the direct sum of $TX$ (since the map $f_*$ is injective on each tangent space) and of the normal line bundle $NX$. This last bundle is trivial together with its induced connection, since it admits a parallel section (the unit normal vector field compatible with the orientations on $TX$ and $T\rz^4$). Moreover, the connection induced on $TX$ by projecting the trivial connection on $f^*T\rz^4$ is the Levi-Civita connection. Hence for every trivialization of $TX$ we obtain by Theorem \ref{tm} that 
$\cs(\omega)\in \Lambda^3_\zz(X)$. By pairing with the fundamental homology class $[X]\in H_3(X,\zz)$ we get the conclusion in the case where $f$ is isometric.

In general, replace $g$ by the pull-back of the euclidean metric through $f$. Since $f$ is conformal, this metric is conformal to $g$, so the Chern--Simons invariant does not change.
\end{proof}

A notion more general than isometric immersions is that of \emph{equiaffine} immersions, see \cite{NS}. Recall that a connection on the tangent bundle of a manifold $X^n$ is called equiaffine (or unimodular) if it preserves a volume form. Equivalently, an equiaffine connection induces the trivial connection on the top form bundle $\Lambda^nX$. Yet another way of defining it would be to ask that the global holonomy group be a subgroup of $\mathrm{SL_n}$. An equiaffine immersion of a manifold endowed with a torsion-free equiaffine connection $\nabla^{TX}$ is by definition an immersion $f:X\to \rz^m$ together with a complement $E'$ of $TX$ inside $f^*T\rz^m$, such that for the projector $\Pi_{TX}$ defined by $E'$, the connection on $TX$ induced from $f^*\nabla^{^{\rz^4}}$ is precisely $\nabla$:
\[\nabla=\Pi_{TX}f^*\nabla^{\rz^4}\Pi_{TX}.
\]
A generalization of the result of Chern and Simons was recently found for equiaffine immersions:

\begin{theorem}[\v{C}ap-Flood-Mettler \cite{cfm}]\label{corc}
Let $X^3$ be a compact oriented $3$-fold with a torsion-free connection $\nabla$ on $TX$ which preserves a volume form. If $(X,\nabla)$ admits an equiaffine immersion in $\rz^4$ then 
\[\int_X\cs(\nabla)\in\zz.\]
\end{theorem}
The proof in \cite{cfm} is based on a notion of ``flat embeddings'' of principal bundles with structure group $G$ inside some larger Lie group $\tilde{G}$. Theorem \ref{corc} is in fact a particular case of a more general vanishing result:

\begin{theorem}\label{tma}
Let $X^n$ be a smooth manifold, and $\nabla^X$ a connection on $TX$. Assume that:
\begin{enumerate}
\item the local holonomy of $\nabla^X$ is a subgroup of $\SL_n(\rz)$;
\item there exists an immersion $f:X\to Y$ in a manifold $Y^{n+1}$ and a trivial connection $\nabla^Y$ on $TY$;
\item there exists a rank-$1$ complement $E$ of $TX$ inside $f^*TY$ such that 
\[\nabla^X=\Pi_{TX}f^*(\nabla^{Y})\Pi_{TX},
\]
where $\Pi_{TX}:f^*TY\to TX$ is the projection on the first factor in $f^*TY=TX\oplus E$.
\end{enumerate}
Then the Chern--Simons invariant $\CS(\nabla^X)$ vanishes.
\end{theorem}
\begin{proof}
The connection induced by $\nabla^X$ on the form bundle $\Lambda^n X$ has holonomy in the group $\{1\}$, hence it is flat and the line bundle $\Lambda^n X$ is trivial. 
Let thus $\mu^X$ be a parallel volume form on $X$ for $\nabla^X$, and choose a nonzero section $e$ in $E$ such that $\mu^X\wedge e$ is parallel for the trivial connection $f^*(\nabla^{Y})$ on $\Lambda^{n+1}f^*TY$. (The section $v_{n+1}$ is unique up to a locally constant multiplicative factor.) Since $\nabla^X=\Pi_{TX}f^*(\nabla^{Y})\Pi_{TX}$, it follows that $\Pi_{E}f^*\nabla^{Y}e=0$,
implying that the induced connection on the line bundle $E$ 
\[\nabla^E=\Pi_{E}f^*(\nabla^{Y})\Pi_{E}\] 
is flat. The conclusion follows now from Theorem \ref{tm}.
\end{proof}
Theorem \ref{corc} follows from Theorem \ref{tma} by taking $X$ to be compact, $3$-dimensional and oriented (so in particular parallelizable), $Y=\rz^4$ equipped with is standard flat connection, and $E$ the normal line bundle. Since in Theorem \ref{corc} the connection $\nabla^X$ is assumed to preserve a global volume form, the \emph{global} holonomy group of $\nabla^X$ must be a subgroup of $\SL_3(\rz)$. On an oriented $3$-manifold $X$, a $3$-form belongs to
$\Lambda^3_\zz(X)$ if and only if its integral on $X$ belongs to $\zz$.

Note that the method of proof in \cite{cfm} relies on mapping the principal frame bundle of $X$ into a principal bundle with total space a larger Lie group; this method does not generalise in an obvious way to the setting of Theorem \ref{tma}.

\section{Applications to conformal immersions in translation manifolds}\label{aci}

We conclude this note with an obstruction to conformal immersions in translation manifolds.
\subsection{Translation manifolds}
By definition, a \emph{translation manifold structure} on a topological manifold $Y$ denotes an atlas 
$\mathcal T$ whose transition maps consist of Euclidean translations. Since translations are oriented isometries of $\rz^n$, $Y$ is oriented and inherits from the Euclidean space a flat Riemannian metric $g^Y$. This metric has trivial holonomy on $Y$, a parallel global frame being defined at any $p\in Y$ by the preimage through any chart in $\mathcal T$ near $p$ of the standard frame on $\rz^n$. Since translations preserve the standard frame, this is well-defined independently of the chart in $\mathcal T$.
Conversely, a Riemannian manifold with trivial global holonomy must in particular be parallelizable (because one can fix a global parallel frame) and locally isometric to $\rz^n$ (since flat). Out of all local isometries with $\rz^n$, select those that map the fixed parallel frame onto the standard frame of $\rz^n$, obtaining in this way a translation manifold structure.

In real dimension $2$ there exists a related notion of \emph{translation surface}, which means a compact Riemann surface endowed with a holomorphic $1$-form $\alpha$. Outside the zero locus of $\alpha$, the translation surface has a natural translation atlas whose charts consist of local primitives of $\alpha$, and a flat metric defined by $\Re(\alpha\otimes \overline{\alpha})$, with conical singularities at the zero locus, with angle an integer multiple of $2\pi$. According to our definition, the complement of the zero locus of $\alpha$ in a translation surface is a translation manifold of dimension $2$. Note, however, that a general translation manifold cannot always be compactified by adding some conical points. 
Note also that translation manifolds are oriented and flat, but the converse does not hold. 
A free quotient of $\rz^n$ by a discrete subgroup of isometries is a translation manifold if and only if all group elements are translations. In general, a translation manifold is not required to be complete, so it need not immerse in some quotient of $\rz^n$.

\subsection{Obstructions to conformal immersions in translation manifolds}
\begin{proof}[Proof of Theorem \ref{gencsd}]
By Proposition \ref{confinv}, the Chern--Simons form in a fixed frame is conformally invariant up to exact forms on $X$, so we can assume that the metric on $X$ is chosen such that the immersion in $Y$ is isometric.
Having trivial holonomy means that we can find a global frame on $Y$ parallel for $\nabla^Y$; clearly in this case the Chern--Simons invariant of $\nabla^Y$ vanishes (Lemma \ref{lvan}). By Theorem \ref{tma}, $\CS(\nabla^X)\equiv 0\mod \Lambda^3_\zz(X)$ so by integration along $X$ we get an integer.
\end{proof}

In dimension $n=3$ we obtain:
\begin{corollary}\label{gencs}
If an oriented Riemannian compact $3$-fold 
$(X,g)$ immerses in a $4$-dimensional translation manifold, then 
\[\int_X\CS(\nabla^g)\in\zz.\]
\end{corollary}
\subsection{Examples}
The Chern--Simons invariant of the real projective space $\rz P^3=\SO(3)$ equals $1/2+\zz$. More generally, the Riemannian Chern--Simons invariant of the lens space $L_{p;a,b}$ (see Appendix 
\ref{Appendix2}) with its spherical metric equals $-1/p$ (Proposition \ref{csls}).
Lens spaces are locally isometric to $\sz^3\hookrightarrow \rz^4$. They also admit immersions in $\rz^4$, which however do not respect the conformal class. 
Despite these facts, it follows from Corollary \ref{gencs} and the computation in Appendix \ref{Appendix2}
that lens spaces cannot be conformally immersed in any Riemannian $4$-manifold with trivial holonomy.

\appendix
\section{De Rham cohomology of the general linear group}\label{apa}

We relegated to this section the proof of the fact, crucial for the definition of Chern--Simons invariants, that the closed $3$-form $\frac{1}{48\pi^2}\tr(\mu^3)$, constructed from the Maurer--Cartan form on $\GL_n^+(\rz)$, is integral for every $n$. (This is standard for $n=3$.) We prove at the same time the analogous result for the form $\frac{1}{24\pi^2}\tr_\cz(\mu^3)$ on $\GLnC$.

\subsection{Retraction onto the maximal compact subgroups}
The general linear group $\GLnC$ retracts onto $\Un$ by the deformation
\begin{align*}
\Phi:[0,1]\times \GLnC\to \GLnC,&& \Phi_s(A)=(AA^*)^{-s/2}A.
\end{align*}
The restriction of $\Phi$ to $\GLnR^+$ defines a deformation retraction onto $\SO(n)$. Thus the inclusions $\Un\hookrightarrow \GLnC$ and $\SO(n)\hookrightarrow \GLnR^+$ induce isomorphisms in (co)homology.

\subsection{Cohomology classes constructed from the Maurer--Cartan form}
Let $G\subset \GLnR$ be a linear Lie group. The Maurer--Cartan $1$-form $\mu=g^{-1}\dd g \in\Lambda^1(G)\otimes \gl_n\rz$ is defined by $\mu_g(V)=g^{-1}V$ for $g\in G, V\in T_gG$. It satisfies the equation
\[\dd\mu=\mu^2
\]
which is equivalent to the vanishing of the curvature of the unique connection on the trivial $G$-bundle over a point. It follows that the form $\tr(\mu^3)$ is closed:
\[\dd\tr(\mu^3) = \tr\left(\dd(\mu^3)\right) = \tr\left(\dd(\mu)\mu^2)-\mu\dd(\mu)\mu+\mu^2\dd(\mu)\right) = 
3\tr\left(\dd(\mu)\mu^2)\right) = 3\tr(\mu^4) = 0.
\]
The last equality holds because on one hand, 
$\mu^4=\mu^3\mu=\mu\mu^3$, while on the other hand, from the trace property, $\tr(\mu^3\mu)=-\tr(\mu\mu^3)$.

The case where $G$ is a subgroup in $\GLnC$ is entirely similar, the only difference being that in that case, $\tr_\cz(\mu^3)$ is a complex-valued closed $3$-form on $G$.

Note that the Maurer--Cartan form $\mu$ on $G\subset \GLnR$ coincides with the pullback to $G$ of the Maurer--Cartan form of the group $\GLnR$ itself. A similar statement holds for complex linear groups.

\subsection{Right regular representation of the quaternion algebra}
The field $\kz$ of quaternions is isomorphic  to $\cz^2$ as complex vector spaces, by the identification $z_1+z_2j\leftrightarrow (z_1,z_2)$.
This isomorphism is an isometry when we endow $\kz$ with the Hermitian inner product
\[\langle q_1,q_2\rangle_\cz =\mathfrak{C}(q_1\overline{q_2}),
\]
where $\mathfrak{C}(z_1+z_2j):=z_1$ denotes the complex part of the quaternion $z_1+z_2j$.
$\kz$ is also isometric to $\rz^4$ as real vector spaces 
by the identification $(a,b,c,d)\leftrightarrow a+bi+cj+dk$, with respect to the scalar product
$\langle q_1,q_2\rangle =\Re(q_1\overline{q_2})$.

Define a $\rz$-algebra morphism
\begin{align*}
R:\kz\to \End_\kz(\kz),&&R(q)v:=v\overline{q}.
\end{align*}
By composing $R$ with the chain of inclusion maps $\End_\kz(\kz)\hookrightarrow \End_\cz(\kz)\hookrightarrow \End_\rz(\kz)$, we get $\rz$-algebra maps $R_\cz:\kz\to \End_\cz(\kz)$, respectively
$R_\rz:\kz\to \End_\rz(\kz)$.
Since $R(q)$ is an isometry for $|q|=1$, by restricting to the Lie group of unit-length quaternions
we get Lie group morphisms
\begin{align*}
R_\cz:\sz^3\to \U(2),&&
R_\rz:\sz^3\to\mathrm{O}(4),&&R_\cz(q)=R_\rz(q)=(v\mapsto v\overline{q}).
\end{align*}
Furthermore, the group morphism $\det\circ R_\cz:\sz^3\to \sz^1$ factors through the abelianization of $\sz^3$, which is $\{\pm 1\}$. Since $\sz^3$ is connected, it follows that $R_\cz$, respectively 
$R_\rz$, take values in $\SU(2)$, respectively $\SO(4)$. In fact $R_\cz:\sz^3\to \SU(2)$ turns out to be an isomorphism, while $R_\rz:\sz^3\hookrightarrow\SO(4)$ is an embedding.

\subsection{Homology of $\U(n)$ and $\SO(n)$}
Let $e_1$ denote the first element in the standard basis of $\rz^n$. For every $n\geq 2$ consider the principal fibre bundle with structure group 
$\SO(n-1)=\Stab(e_1)$
\begin{align}\label{fson}
\xymatrix{
\SO(n)\ar[d]^\pi \ar@{-}[r]& \SO(n-1)\\
\sz^{n-1}&
}
\end{align}
where $\pi(A)=Ae_1$. When $n=4$, the embedding
\begin{align}\label{sectr}
s:\sz^3\to \SO(4),&&s(q)=R_\rz(q^{-1})
\end{align}
is a global section in the fibration $\pi$, giving rise to a product decomposition of manifolds
\[\SO(4)=s(\sz^3)\cdot\SO(3).
\]
By the K\"unneth formula, $H_3(\SO(4),\zz)$ is hence generated by the two $3$-cycles of the subgroups $s(\sz^3)$ and 
$\SO(3)$, the latter being embedded in $\SO(4)$ in the lower right corner. 

\begin{lemma}\label{isu}
For $n\geq 5$, the map 
$H_3(\SO(n-1),\zz)\to H_3(\SO(n),\zz)$ induced by inclusion is surjective. 
\end{lemma}
\begin{proof}
Represent any $3$-homology class on $\SO(n)$ by a smooth singular $3$-cycle $C$. By dimensional reasons, 
$\pi\circ C$ must avoid at least a point $p\in\sz^{n-1}$. Let $(\Phi_t)_{0\leq t\leq 1}$ be a deformation retract of $\sz\setminus\{p\}$ onto ${-p}$. Using a connection in the fibration \eqref{fson}, we lift the retraction horizontally, to define a retraction by deformation of $\pi^{-1}(\sz\setminus\{p\})$ to the fiber over $-p$, hence to $\SO(n-1)$. It follows that the cycle $C$ is cohomologous to a cycle in $\SO(n-1)\subset \SO(n)$, as claimed.
\end{proof}

The case of $\GLnC$ is similar but simpler, using the principal fibration
\begin{align}\label{fun}
\xymatrix{
\U(n)\ar[d]^\pi \ar@{-}[r]& \U(n-1)\\
\sz^{2n-1}&
}
\end{align}
For $n=2$ this fibration is trivial because it admits the section $s_\cz=R_\cz\circ \mathrm{inv}$ as in \eqref{sectr}, inducing a product decomposition
\[\U(2)=s_\cz(\sz^3)\cdot\U(1),
\]
hence $H_3(\U(2),\zz)$ is infinite cyclic, with generator the cycle $s(\sz^3)$. For larger $n$, the lacunarity of the Leray-Serre homology spectral sequence of the fibration \eqref{fun} shows that the inclusion 
$\U(n-1)\hookrightarrow \U(n)$ induces an isomorphism in $H_3$, so for every $n$ the $3$-homology of $\U(n)$ is generated by the cycle $\sz^3$.

\subsection{Integrality}
\begin{lemma}
For every $N$, the form $\frac{1}{48\pi^2}\tr(\mu^3)$ belongs to $\Lambda^3_\zz(\GL^+_n(\rz))$.
\end{lemma}
\begin{proof}

Since $\SO(n)$ is a deformation retract of $\GL^+_n(\rz)$, we have to check that the evaluation of $\frac{1}{48\pi^2}\tr(\mu^3)$ on every $3$-cycle with integer coefficients
in $\SO(n)$ is integral. This is trivially true (for dimensional reasons) when $n\leq 2$.

For every $n\geq 3$ we have the cycle consisting of (a triangulation of) the subgroup $\SO(3)\hookrightarrow \SO(n)$ embedded diagonally in the right lower corner. One checks immediately that
the restriction of $\mu$ to $\SO(3)$ is the Maurer--Cartan form of that subgroup, composed with the diagonal inclusion of spaces of matrices $\gl(3)\hookrightarrow \gl(n)$. This inclusion is moreover compatible with the trace maps. We now claim that
\begin{equation}\label{intso}
\tr(\mu^3)_{|\SO(3)}=-48\vol,
\end{equation}
where $\vol$ is the standard volume form on $\sz^3$.
To see this, consider the adjoint representation
\begin{align}\label{stso}
\Ad: \sz^3\to \GL(\kz), && \Ad_q(v) = qvq^{-1}.
\end{align}
It decomposes into the trivial representation on $\rz$ (which is the center of $\kz$) and a orthogonal representation on the span over $\rz$ of the imaginary quaternions 
$\{i,j,k\}$, identified with $\rz^3$, that we denote by $\Ad: \sz^3\to\SO(3)$. The group morphism $\Ad$ is surjective, a $2:1$ covering map with kernel $\{\pm 1\}$, inducing an isomorphism $\ad:\rz^3\to \so(3)$ at the level of Lie algebras. The group $\sz^3\subset \kz$ with its standard bi-invariant metric admits an orthonormal frame $(I,J,K)$ of left-invariant vector fields defined at $q\in \sz^3$ by $I_q=qi\in T_q \sz^3\subset \kz$, etc. In this basis, the map $\ad$ takes the form
\begin{align*}
\ad(I)=\begin{bmatrix} 0&0&0\\0&0&-2\\0&2&0
\end{bmatrix},&&
\ad(J)=\begin{bmatrix} 0&0&2\\0&0&0\\-2&0&0
\end{bmatrix},&&
\ad(K)=\begin{bmatrix} 0&-2&0\\2&0&0\\0&0&0
\end{bmatrix}.
\end{align*}
From this, we compute
\begin{align*}
\tr(\ad(I)\ad(J)\ad(K))=-8,&&\tr(\ad(I)\ad(K)\ad(J))=8.
\end{align*}
Note that $\tr(\mu^3)$ is left-invariant, so it is a multiple (to be determined) of the standard volume form induced from $\sz^3$ by the $2:1$ covering map \eqref{stso}. Moreover, from the above identities, $\Ad^*\tr(\mu^3)(I,J,K)=48$.  Together, these two facts end the proof of \eqref{intso}.

When $n=3$, the cycle $\SO(3)$ clearly generates $H_3(\SO(3),\zz)$.

For $n=4$, we have seen that
$\SO(4)$ is diffeomorphic to the Cartesian product of the submanifolds $\SO(3)$ and $s(\sz^3)$, so $H_3(\SO(4),\zz)$ is free abelian with \emph{two} generators, the extra one being the subgroup $s(\sz^3)$, or equivalently\footnote{The $3$-cycles $s(\sz^3)$ and 
$R_\rz(\sz^3)$ in $\SO(4)$ differ by the orientation-changing involution $q\mapsto q^{-1}$ of $\sz^3$.}  $R_\rz(\sz^3)$. 
The form $R_\rz^*\tr(\mu^3)$ is left-invariant, hence a constant multiple of the standard volume form. Write 
$R_\rz^*\mu\in\Lambda^1(\sz^3,\gl_4(\rz))$ in the standard basis of $\kz$ as
$R_\rz^*\mu(V)=R_\rz(V)$ for every $V\in T\sz^3$. Since $R_\rz$ is multiplicative, for every permutation $\sigma\in\Sigma_3$, 
\[R_\rz(\sigma(I))R_\rz(\sigma(J))R_\rz(\sigma(K))=R_\rz(\sigma(I)\sigma(J)\sigma(K))=R_\rz(-\mathrm{sign}(\sigma))=-\mathrm{sign}(\sigma)\mathrm{Id}_4,\] 
where 
$\mathrm{sign}(\sigma)$ is the signature of the permutation $\sigma$,
therefore $\tr\left((R_\rz^*\mu^3)(I,J,K)\right)=-24$, showing that 
\begin{equation}\label{ints}
\tr(\mu^3)_{|R_\rz(\sz^3)}=-24\vol.
\end{equation}

Finally, for $n\geq 5$, by Lemma \ref{isu} the cycles $\SO(3)$ and $s(\sz^3)$ still generate\footnote{Albeit not freely, but this is irrelevant for the sake of our argument.} $H_3(\SO(n),\zz)$. Since $\vol(\sz^3)=2\pi^2$ and $\vol(\SO(3))=\pi^2$, it follows from \eqref{intso} and \eqref{ints} that
$\frac{1}{48\pi^2}\tr(\mu^3)$ evaluated on these cycles equals $- 1$, respectively $1$.
\end{proof}

\begin{lemma}\label{lemintc}
For every $n$, the form $\frac{1}{24\pi^2}\tr_\cz(\mu^3)$ belongs to $\Lambda^3_\zz(\GLnC)$.
\end{lemma}
\begin{proof}
This is similar to the real case using the same arguments as above, only simpler since for every $n$ the homology space $H_3(\GLnC)$ is free of rank $1$, generated by the cycle $\sz^3\to\SU(2)\hookrightarrow \GLnC$. The matrix-valued form $\mu^3$ evaluated on the orthonormal frame $\{I,J,K\}$ equals $3!$ times $R_\cz(I)R_\cz(J)R_\cz(K)=R_\cz(IJK)=R_\cz(-1)=-\mathrm{Id}_2$, so 
$\tr_\cz(\mu^3)(I,J,K)=12$. Since $\vol(\sz^3)=2\pi^2$, it follows that
\[\int_{\sz^3} \tr_\cz(\mu^3) = -24\pi^2.\qedhere
\]
\end{proof}

\section{Riemannian Chern--Simons invariant of lens spaces}\label{Appendix2}

Chern--Simons invariants of flat $\SU(2)$ connections on $3$-dimensional lens spaces have been computed by Kirk-Klassen \cite{kirkk}. Here we compute the real Chern--Simons invariant of the Levi-Civita connection.

\subsection{The sphere $\sz^3$}
The Chern--Simons form of the Levi-Civita connection $\nabla$ on $\sz^3$ in the orthonormal frame 
$(I,J,K)$ was computed in the original paper of Chern and Simons \cite{cs}:
\[\cs(\nabla;(I,J,K))=-\frac{1}{2}\vol.
\] 
We refer e.g.\ to \cite{cotton} for the easy proof, which can also be derived from the results of the previous sections. As a corollary, $\CS(\nabla)=0\mod \zz$.
\subsection{Lens spaces}
The \emph{lens space} $L_{p;q_1,q_2}$ is defined as the quotient of $\sz^3\subset \cz^2$ by the cyclic group $\Gamma$ of order $p$ generated by the isometry
\[(z_1,z_2)\mapsto (e^{2\pi iq_1/p}z_1,e^{2\pi i q_2/p}z_2)
\]
for some integers $p,q_1,q_2$. If $q_1,q_2$ are both relatively prime to $p$, the action is free, so the lens space $ L_{p;q_1,q_2}$ is an oriented compact Riemannian $3$-fold with constant sectional curvature $1$.
Denote by $\pi:\sz^3\to L_{p;q_1,q_2}$ the natural $p:1$ covering map.
\begin{proposition}\label{csls}
The Chern--Simons invariant of the Levi-Civita connection for the standard spherical metric on $L_{p;q_1,q_2}$ equals $-\frac{1}{p}\mod \zz$.
\end{proposition} 
\begin{proof}
The vector field $I$ defined by right multiplication with $i\in\kz$ is invariant by $\Gamma$, so its orthogonal complement, the span of $J$ and $K$, forms an invariant oriented real vector bundle of rank $2$ with inner product, which can therefore be naturally equipped with a complex line bundle structure. Being $\Gamma$-invariant, this bundle descends to a complex line bundle $E$ over $L_{p;q_1,q_2}$. Since 
$H^2(L_{p;q_1,q_2},\zz)=0$, the Chern class of $E$ vanishes so $E$ is trivial. A trivialization of $E$ by a global section $U$ of length $1$ defines a trivialization of the underlying real vector bundle by an orthonormal frame $(U,V)$, where $V$ is the rotation of $U$ by angle $\pi/2$. Together with $I$, these vector fields form an orthonormal frame on $L_{p;q_1,q_2}$. Let $a:\sz^3\to\SO(3)$ be the matrix relating the pull-back of the frame $(I,U,V)$ to $\sz^3$, denoted 
$(\pi^*I,\pi^*U,\pi^*V)$, and $(I,J,K)$. Since $I=\pi^*I$ is the first vector in both frames, the matrix $a$ relating them is block-diagonal, orthogonal, with $1$ in the upper left corner, hence a rotation in the bottom right corner. In other words, $a$ factors through $\SO(2)$, embedded in $\SO(3)$ as the stabilizer of $e_1$.

Use now formula \eqref{csop} to relate the Chern--Simons forms of the Levi-Civita connection on $\sz^3$ in the two trivializations $(\pi^*I,\pi^*U,\pi^*V)$ and $(I,J,K)$. 

Since the map $a$ factors through the $1$-dimensional subgroup $\SO(2)$, the $3$-form $\tr((a^{-1}\dd a)^3)=a^*\tr(\mu^3)$ vanishes. 
Moreover, by the Stokes formula, $\int_{\sz^3}\dd\tr(\dd a a^{-1}\omega)=0$.
In conclusion, 
\begin{equation}\label{ecs}
\int_{\sz^3} \cs(\nabla;(\pi^*I,\pi^*U,\pi^*V)) = \int_{\sz^3} \cs(\nabla;(I,J,K))=-1.
\end{equation}
Since $\pi$ is an isometric covering map of degree $p$, we have the equality of $3$-forms on $\sz^3$
\[\pi^*\cs(\nabla;(I,U,V))= \cs(\nabla;(\pi^*I,\pi^*U,\pi^*V)).
\]
Hence the left-hand side of \eqref{ecs} equals $p$ times the integral on $L_{p;q_1,q_2}$ of the Chern--Simons form $\cs(\nabla;(I,U,V))$, leading to 
\[\int_{L_{p;q_1,q_2}}\cs(\nabla;(I,U,V))=-\frac{1}{p}.\qedhere
\]
\end{proof}

\subsection*{Acknowledgement} Partially supported from the PNRR project III-C9-2023-I8-CF149.

\end{document}